\documentclass[12pt]{amsart}
\usepackage{amssymb, amsmath, amsfonts, amsthm, esint,graphicx}
\usepackage[abbrev]{amsrefs}
\usepackage[margin=0.8 in]{geometry}
\usepackage[colorlinks=true, linkcolor=blue]{hyperref}

\def \dd {\partial}

\DeclareMathOperator{\Ric}{Ric}

\DeclareMathOperator{\diam}{diam}

\DeclareMathOperator{\diver}{div}

\DeclareMathOperator{\sech}{sech}

\DeclareMathOperator{\Span}{span}

\title{Explicit Lower Bound of the first eigenvalue of the laplacian on K\"ahler manifolds}
\author{Benjamin Rutkowski}
\address{Department of Mathematics\\
         California State University, Fullerton\\
         Fullerton, CA 92831}
\email{\href{mailto:brutkowski0@csu.fullerton.edu}{brutkowski0@csu.fullerton.edu}}
\author{Shoo Seto}
\address{Department of Mathematics\\
         California State University, Fullerton\\
         Fullerton, CA 92831}
\email{\href{mailto:shoseto@fullerton.edu}{shoseto@fullerton.edu}}
\date{}

\theoremstyle{definition}

\newtheorem{example}{Example}[section]
\newtheorem{lemma}{Lemma}[section]
\newtheorem{remark}{Remark}[section]

\newtheorem{theorem}{Theorem}[section]

\newtheorem*{theorem*}{Theorem}

\begin{document}
\maketitle

\begin{abstract}
We establish an explicit lower bound of the first eigenvalue of the Laplacian on K\"ahler manifolds based off the comparison results of Li and Wang \cite{li-wang}. The lower bound will depend on the diameter, dimension, holomorphic sectional curvature and orthogonal Ricci curvature.
\end{abstract}

\section{Introduction}

Studying the spectrum of the Laplace operator and its relation to the underlying geometry has been an active field of research with tremendous amounts of results.  There are numerous applications in areas outside of mathematics such as physics where the Schr\"odinger operator, which arises in quantum mechanics, can be investigated through the spectrum of the Laplacian.  The discrete energy levels of electrons is a direct consequence of the spectral theory of the Laplacian.   Also, in the branch of harmonic analysis, expanding a function as a Fourier series can be seen as an eigenfunction expansion of the Laplacian and the eigenvalues are directly related to the rate of convergence of the series.  In this article, we will focus on these eigenvalues and in particular, on the first or principal eigenvalue.  

The spectrum depends heavily on the underlying domain of the functions which the operator acts on and looking at the eigenvalue problem on domains of Riemannian manifolds gives rise to many interesting results due to the introduction of curvature.  One can study this problem from both a differential geometric point of view and analytic point of view of the equation.  

Lower bounds of the first eigenvalue, or to be more specific, the lower bound of the gap between the first two eigenvalues, called the ``fundamental gap'', and is of particular interest among eigenvalue problems.  Note that depending on the boundary condition, the first eigenvalue can be zero so that the lower bound of the gap is simply the lower bound of the first non-trivial eigenvalue.  For example in physics, the gap represents the `excitation energy' required to reach the first excited state from the ground state.  In mathematics, the first eigenvalue can tell us information on the rigidity of the domain, the isoperimetric inequality, the Poincar\'e inequality, convergence rates of Fourier series, etc.  

In this article, we give an explicit lower bound of the first nonzero eigenvalue on K\"ahler manifolds which are essentially complex analogues of Riemannian manifolds.  Our result will be based off a recent result of Li and Wang \cite{li-wang}, which gives a comparison result of the first eigenvalue of the Laplacian on compact K\"ahler manifolds with a one-dimensional eigenvalue value problem, (see \eqref{li-wang-eqn}.  From their comparison result we can give the following explicit lower bound
\begin{theorem}[Main result]\label{main}
Let $(M^m,g,J)$ be a compact K\"ahler manifold of complex dimension $m$ and diameter $D$ whose holomorphic sectional curvature is bounded from below by $H \geq 4\kappa_1$ and orthogonal Ricci curvature is bounded from below by $\Ric^\perp \geq 2(m-1)\kappa_2$ for some $\kappa_1,\kappa_2 \in \mathbb{R}$.  Let $\lambda_1$ be the first nonzero eigenvalue of the Laplacian on $M$ (with Neumann boundary condition if $M$ has a strictly convex boundary). 
\begin{equation*}
\mu_1 \geq \sup_{s\in (0,1)}\left\{4s(1-s)\frac{\pi^2}{D^2} + 2s(\kappa_2(m-1) + 2\kappa_1\right\}
\end{equation*}
\end{theorem}

Our method is based off of a technique used by Shi-Zhang \cite{shi-zhang} and by Futaki-Li-Li \cite{futaki-li-li} which is reminiscent of the Moser iteration technique and takes advantage of Wirtinger's inequality, which itself is a consequence of the first eigenvalue in one-dimension.

\section{Set up and preliminary results}

\subsection{Laplacian on Riemannian manifolds}
Let $(M,g)$ be an $n$-dimensional Riemannian manifold.   The Laplacian (or Laplace-Beltrami) operator is a second order elliptic operator acting on $C^2(M)$ defined by
\begin{equation*}
\Delta u := \diver(\nabla u)
\end{equation*}
For some point $p\in M$, we let $(x^1,x^2,\ldots, x^n)$ be some local coordinates.  Then in these coordinates the Laplacian can be written as 
\begin{equation*}
\Delta u = \frac{1}{\sqrt{\det(g)}}\frac{\dd}{\dd x^j}\left(\sqrt{\det(g)}g^{ij}\frac{\dd u}{\dd x^i}\right)
\end{equation*}
where $g^{ij}$ is the $(i,j)$-th entry of the inverse of the metric tensor $(g_{ij})$.  We say that $u\not\equiv 0$ is an eigenfunction of the Laplacian on $M$ with eigenvalue $\lambda \in \mathbb{R}$, if
\begin{equation*}
\Delta u = -\lambda u
\end{equation*}
in $M$.  When $\dd M = \emptyset$, from the spectral theory of compact self-adjoint operators, we know that the eigenvalues are discrete and can be arranged as a non-decreasing diverging series
\begin{equation*}
0=\lambda_0 < \lambda_1 \leq \lambda_2 \leq \ldots \to \infty.
\end{equation*}
If $\dd M \neq \emptyset$, we must impose boundary conditions to obtain discreteness of the eigenvalues.  For instance, the Dirichlet boundary condition prescribes the boundary value of the solution, often times being constantly zero.  The Neumann boundary condition prescribes the value of the normal derivative of the solution at the boundary, again often times this being zero.  The Neumann boundary and the closed (compact and with no boundary) case behave similarly since by divergence theorem
\begin{align*}
\int_M \nabla_n udS = \int_{\dd M} udV = 0
\end{align*}
when either $\nabla_nu=0$ or $\dd M=\emptyset$. 

The eigenvalues also satisfy a variational characterization.  Of note is the first (non-trivial) eigenvalue
\begin{equation*}
\lambda_1 = \inf_{u \in C^2(M)}\left\{ \frac{\int_M|\nabla u|^2dV}{\int_M u^2dV} \ \biggr| \ u \not\equiv 0, u \in \mathcal{D}\right\}
\end{equation*}
where $\mathcal{D}$ denotes a class of functions with prescribed behavior which will depend on the type of eigenvalue problem we wish to study.  For example, $\mathcal{D}$ are functions with zero boundary value, provided $\dd M \neq \emptyset$, while for the Neumann and closed case, $\mathcal{D}$ are functions such that $\int_M udV=0$.

\begin{example}
In one-dimension, a compact connected manifold is simply an interval $[0,L]$ for some $L>0$.  The Dirichlet eigenvalue problem is
\begin{align*}
\begin{cases}
u'' + \lambda u = 0 & \text{ on }[0,L]\\
u(0)= u(L) = 0.
\end{cases}
\end{align*}
The eigenfunctions are given by $u_n(x) = \sin(\tfrac{n\pi}{L}x)$, $n\in \mathbb{N}$, with eigenvalues $\lambda_n = \frac{n^2\pi^2}{L^2}$. In particular, the variational characterization of the first eigenvalue gives us \textit{Wirtinger's inequality}.
\begin{lemma}[Wirtinger's inequality]\label{wirt}
For smooth $f$ such that $f(0)=f(L)=0$, we have
\begin{equation*}
\int_0^L f^2 dx \leq \frac{L^2}{\pi^2}\int_0^L (f')^2dx.
\end{equation*}
\end{lemma}
\end{example}

In general, finding the explicit values of eigenvalues is impossible for all but highly symmetric domains.  Instead we opt to find estimates, in particular on lower bounds of the first eigenvalues.  Two classical results are the Lichnerowicz estimate and the Zhong-Yang estimate.
\begin{theorem}[Lichnerowicz \cite{lichnerowicz}]
Let $(M,g)$ be a complete $n$-dimensional Riemannian manifold with $\Ric \geq (n-1)K>0$.  Let $\lambda_1$ be the first nonzero eigenvalue of the Laplacian on $M$.  Then
\begin{equation*}
\lambda_1(M) \geq nK.
\end{equation*}
\end{theorem}
\begin{remark}
Obata \cite{obata} establishes that equality holds if and only if $M$ is isometric to a simply connected space of constant curvature $K$, i.e., the round $n$-sphere of radius $\frac{1}{\sqrt{K}}$.
\end{remark}

\begin{theorem}[Zhong-Yang]
Let $(M,g)$ be a closed $n$-dimensional Riemannian manifold with $\Ric \geq 0$.  Then
\begin{align*}
\lambda_1(M) \geq \frac{\pi^2}{D^2},
\end{align*}
where $D=\diam(M)$.
\end{theorem}
\begin{remark}
Hang and Wang \cite{hang-wang} established that equality holds if and only if $M$ is isometric to $S^1$ with radius $\frac{D}{\pi}$.
\end{remark}

Furthermore, Kroger \cite{kroger} (see also Bakry-Qian \cite{bakry-qian}) gives a uniform approach to the above two results via a comparison method with a one-dimensional eigenvalue problem.
\begin{theorem}[\cite{kroger}]\label{onedim}
Let $(M,g)$ be a compact $n$-dimensional Riemannian manifold (possibly with a smooth convex boundary) with diameter $D$ and $\Ric \geq (n-1)K$ for $K\in \mathbb{R}$.  We assume Neumann boundary conditions if $\dd M \neq \emptyset$.  Then
\begin{align*}
\lambda_1 \geq \bar{\lambda}_1(n,K,D)
\end{align*}
where $\bar{\lambda}_1(n,K,D)$ is the first nonzero Neumann eigenvalue of the one-dimensional eigenvalue problem
\begin{align*}
\varphi''-(n-1)T_K\varphi'=-\bar{\lambda}\varphi
\end{align*}
on the interval $[-D/2,D/2]$.  Here the function $T_K$, $K \in \mathbb{R}$ is defined to be 
\begin{equation*}
T_K(x) := \begin{cases}
\sqrt{K}\tan(\sqrt{K}x), & K >0\\
0 & K = 0\\
-\sqrt{-K}\tanh(\sqrt{-K}x) & K <0.
\end{cases}
\end{equation*}
\end{theorem}
The above is indeed a uniformization of Lichnerowicz and Zhong-Yang estimates since when $K=0$, the explicit solution is given by $\varphi(x) = \cos(\tfrac{\pi}{D}x)$. When $K>0$, we can consider the model when $D=\frac{\pi}{\sqrt{K}}$ so that the eigenfunction is given by $\varphi(x) = \sin(\sqrt{K}x)$ which has the eigenvalue $\bar{\lambda} = nK$.

By interpolating linearly in $K$, the Li conjecture states that $\lambda_1 \geq \frac{\pi^2}{D^2}+(n-1)K$.  However the Li conjecture is false and can be shown by computing an asymptotic expansion of the eigenvalue.  By analyzing the one-dimensional eigenvalue problem in Theorem \ref{onedim}, Shi and Zhang show that instead of $(n-1)$, the lower bound involving both the diameter and curvature is the following.
\begin{theorem}[Shi-Zhang \cite{shi-zhang}]
On compact Riemannian manifold with $\Ric \geq (n-1)K$, $\diam(M)=D$, $K\in \mathbb{R}$, the first non-zero eigenvalue of the Laplacian satisfies
\begin{equation*}
\lambda_1 \geq 4(s-s^2)\frac{\pi^2}{D^2}+s(n-1)K
\end{equation*}
for $0< s < 1$. 
\end{theorem}
By choosing $s=\frac{1}{2}$, we get $\lambda_1 \geq \frac{\pi^2}{D^2}+\frac{(n-1)}{2}K$.   See also \cite{futaki-li-li} for an extension of this result to the setting of Bakry-\'Emery manifolds and the Witten-Laplacian.

\subsection{K\"ahler manifolds}
Let $(M,g,J)$ be a complex manifold of complex dimension $m$.  $M$ is called a \textit{K\"ahler} manifold if the K\"ahler form defined by $\omega(X,Y) = g(JX,Y)$ is a closed form, i.e. $d\omega = 0$.  Here $J$ is the complex structure of $J$.  Moreover, on K\"ahler manifolds, we have several curvature operators which arise using the complex structure.  We say a plane $P\subset T_pM$ is holomorphic if it is invariant under the complex structure $J$.  Suppose $P=\Span\{X,JX\}$ for some $X \in T_pM$.  The \textit{holomorphic sectional curvature of $P$}
\begin{align*}
H(P) :=H(X)= \frac{R(X,JX,X,JX)}{|X|^4}.
\end{align*}
We say the holomorphic sectional curvature is bounded from below by $K\in \mathbb{R}$, written $H\geq K$, if $H(P) \geq K$ for all holomorphic planes $P\subset T_pM$ and all $p \in M$. 

The \textit{orthogonal Ricci curvature} $\Ric^\perp$ is defined for any $X \in T_pM$ by
\begin{align*}
\Ric^\perp(X,X) = \Ric(X,X)-H(X)|X|^2.
\end{align*}

Ni and Zheng studied this curvature in \cite{ni-zheng} and in particular proved a Laplace comparison theorem under an orthogonal Ricci curvature lower bound.  


With the additional K\"ahler structure, the eigenvalue estimates of the Laplacian can be improved.  This was observed by Lichnerowicz \cite{lichnerowicz} and in particular, the Lichnerowicz estimate becomes

\begin{theorem}[K\"ahler Lichnerowicz estimate \cite{lichnerowicz}]
Let $(M,g,J)$ be an $n=2m$ dimensional closed K\"ahler manifold with $\Ric \geq (n-1)K>0$.  Then 
\begin{align*}
\lambda_1 \geq 2(n-1)K.
\end{align*}
\end{theorem}
A proof of this can be found for instance in \cite{blacker-seto} or \cite{guedj}.  Recently, Li and Wang established the K\"ahler analogue of Theorem \ref{onedim} which reflects more on the K\"ahler structure by imposing conditions on the holomorphic sectional and orthogonal Ricci curvature.

\begin{theorem}\label{liwang}[Li-Wang \cite{li-wang}]
Let $(M^m,g,J)$ be a compact K\"ahler manifold of complex dimension $m$ and diameter $D$ whose holomorphic sectional curvature is bounded from below by $H \geq 4\kappa_1$ and orthogonal Ricci curvature is bounded from below by $\Ric^\perp \geq 2(m-1)\kappa_2$ for some $\kappa_1,\kappa_2 \in \mathbb{R}$.  Let $\lambda_1$ be the first nonzero eigenvalue of the Laplacian on $M$ (with Neumann boundary condition if $M$ has a strictly convex boundary).  Then
\begin{align*}
\lambda_1 \geq \bar{\lambda}_1(m,\kappa_1,\kappa_2,D),
\end{align*}
where $\bar{\lambda}_1(m,\kappa_1,\kappa_2,D)$ is the first Neumann eigenvalue of the one-dimensional eigenvalue problem
\begin{equation}\label{li-wang-eqn}
\varphi''-(2(m-1)T_{\kappa_2}+T_{4\kappa_1})\varphi'=-\mu\varphi
\end{equation}
on $[-D/2,D/2]$.  
\end{theorem}

\section{Computation of the lower bound}
\subsection{Proof of Theorem \ref{main}}
We first define the function $S_K$. Let 
\begin{align*}
    S_K(x) := 
    \begin{cases}
    K \sec^2(\sqrt{K}x) & K > 0 \\
    0 & K = 0 \\
    K \sech^2(\sqrt{-K}x) & K < 0.
    \end{cases}
\end{align*}
Verify that $(T_K)' = S_K$. Now we establish the following lemmas.
\begin{lemma}
\label{int part}
Let $v$ be a function such that $v(-\frac{D}{2}) = v(\frac{D}{2}) = 0$. Then for any $K\in \mathbb{R}$ and $a>1$, 
\begin{align*} 
    \int_{-D/2}^{D/2} T_K \cdot v^{a-1}v' = -\frac{1}{a}\int_{-D/2}^{D/2} \text{S}_K \cdot (v^\frac{a}{2})^2.
\end{align*}
\end{lemma}

For the remainder of this section, every integral will be evaluated from $-\frac{D}{2}$ to $\frac{D}{2}$. Thus for the sake of brevity, we will suppress the bounds of each integral.

\begin{proof}
\begin{align*}
    \int T_K \cdot v^{a-1}v' & = \frac{1}{a}\int T_k \cdot (v^a)' \\
    & = -\frac{1}{a}\int (T_K)'\cdot v^a \\ 
    & = -\frac{1}{a}\int S_K \cdot (v^\frac{a}{2})^2.
\end{align*} 
The second line follows from integration by parts. The boundary term of integration by parts vanishes because of $v = 0$ on the boundary of $[-\frac{D}{2},\frac{D}{2}]$. 
\end{proof}

\begin{lemma}
\label{bounds}
Let $w$ be any function, then for any $K$
\begin{align*}
    \int S_K \cdot w^2 \geq K\int w^2.
\end{align*}
\end{lemma}
\begin{proof} We first show that $S_K \geq K$.

 Case 1. $K > 0$. Recall $\sec^2(\sqrt{K}x) \geq 1$. Then 
\begin{align*}
S_K = K \sec^2(\sqrt{K}x) \geq K.    
\end{align*}

Case 2. $K=0$ is immediate.

Case 3. $K < 0$. Recall $\sech^2(\sqrt{K}x) \leq 1$. Then 
\begin{align*}
S_K = K \sech^2(\sqrt{K}x) \geq K.    
\end{align*}
The lemma follows from $S_k\cdot w^2 \geq Kw^2$.
\end{proof}

By Theorem \ref{liwang}, it suffices to find a lower bound for $\mu$ in the ODE 
\begin{align*}
    \begin{cases}
    \phi'' - (2(m-1)T_{\kappa_2} + T_{4\kappa_1})\phi' = -\mu \phi & 
    \text{ on } [-\frac{D}{2}, \frac{D}{2}] \\
    \phi'(-\frac{D}{2}) = \phi'(\frac{D}{2}) = 0.
    \end{cases}
\end{align*}
If we expand the ODE and take a single derivative of both sides we get 
\begin{align*}
    \phi''' - 2(m-1)\Big{[}(T_{\kappa_2})'\cdot \phi' + T_{\kappa_2}\cdot \phi''\Big{]} - \Big{[}(T_{4\kappa_1})'\cdot \phi' + T_{4\kappa_1}\cdot \phi''\Big{]} & = -\mu \phi'
\end{align*}
which is equivalent to 
\begin{align*}
    \phi''' - 2(m-1)T_{\kappa_2}\cdot \phi'' - T_{4\kappa_1}\cdot \phi'' & =
    (T_{\kappa_2})'\cdot \phi' + (T_{4\kappa_1})'\cdot \phi' -\mu \phi'.
\end{align*}
Now let $v = \phi'$. Then $v$ solves the ODE with Dirichlet conditions
\begin{equation}\label{ode}
    \begin{cases}
    v'' - 2(m-1)T_{\kappa_2}\cdot v' - T_{4\kappa_1}\cdot v' =
    (T_{\kappa_2})'\cdot v + (T_{4\kappa_1})'\cdot v -\mu v 
    & \text{ on } [-\frac{D}{2}, \frac{D}{2}]\\
    v(-\frac{D}{2}) = v(\frac{D}{2}) = 0.
    \end{cases}
\end{equation}

Multiply both sides of the ODE in \eqref{ode} by $v^{a-1}$ for some $a>1$ and integrate both sides from $-\frac{D}{2}$ to $\frac{D}{2}$. Then
\begin{multline}\label{inteq1}
    \int v^{a-1}v'' -2(m-1)\int T_{\kappa_2}\cdot v^{a-1}v' - \int T_{4\kappa_1}\cdot v^{a-1}v' \\ = 2(m-1)\int (T_{\kappa_2})'\cdot v^a + \int (T_{4\kappa_1})'\cdot v^a - \mu\int v^a.
\end{multline}

Consider the first term of this equation. We can use integration by parts to get
\begin{equation}\label{inteq2}
\begin{aligned}
    \int v^{a-1}v'' & = -(a-1)\int v^{a-2}(v')^2 \\
    & = -(a-1)\int (v^{\frac{a}{2} - 1}v')^2 \\
    & = -\frac{4(a-1)}{a^2}\int ((v^{\frac{a}{2}})')^2.
\end{aligned}
\end{equation}
In the first line, the boundary terms of the integration by parts vanish because of the Dirichlet conditions in \eqref{ode}. 

Now from \eqref{inteq1}, we use \eqref{inteq2}, Lemma \ref{int part} and the fact that $(T_K)' = S_K$ to get

\begin{multline*}
    -\frac{4(a-1)}{a^2}\int ((v^{\frac{a}{2}})')^2 + \frac{2(m-1)}{a}\int S_{\kappa_2}\cdot (v^\frac{a}{2})^2 + \frac{1}{a}\int S_{4\kappa_1}\cdot (v^\frac{a}{2})^2 \\
    = 2(m-1) \int S_{\kappa_2}\cdot (v^\frac{a}{2})^2 + \int S_{4\kappa_1}\cdot (v^\frac{a}{2})^2  - \mu\int (v^\frac{a}{2})^2.
\end{multline*}
Set $w = v^\frac{a}{2}$ and combine terms to get 
\begin{equation}\label{itereq}
    -\frac{4(a-1)}{a^2}\int (w')^2 = \frac{2(m-1)(a-1)}{a}\int S_{\kappa_2}\cdot w^2 + \frac{(a-1)}{a}\int S_{4\kappa_1}\cdot w^2 - \mu\int w^2.
\end{equation}

Note that the coefficients to the first two integrals on the right hand side of \eqref{itereq} are positive.  Then by using Lemma \ref{bounds} and combining like terms we get 
\begin{align*}
    -\frac{4(a-1)}{a^2}\int (w')^2 \geq 
    \Big{(} \frac{2\kappa_2 (m-1)(a-1) + 4\kappa_1 (a-1)}{a} - \mu\Big{)}\int w^2
\end{align*}
Now multiply both sides by $-1$ and rearrange to get. 
\begin{equation}\label{rayleigh}
    \frac{\int (w')^2}{\int w^2} \leq \frac{\mu - \frac{2\kappa_2 (m-1)(a-1) + 4\kappa_1 (a-1)}{a}}{\frac{4(a-1)}{a^2}}.
\end{equation}

Note that because $v$ satisfies the Dirichlet condition in \eqref{ode} we also have that $w(-\frac{D}{2}) = w(\frac{D}{2}) = 0$ by our definition of $w$. Thus, we can use Wirtinger's inequality (Lemma \ref{wirt}) on the left hand side of \eqref{rayleigh} to get 
\begin{align*}
    \frac{\pi^2}{D^2} \leq \frac{\mu - \frac{2\kappa_2 (m-1)(a-1) + 4\kappa_1 (a-1)}{a}}{\frac{4(a-1)}{a^2}}.
\end{align*}
Solving for $\mu$ and simplifying, we get
\begin{align*}
    \mu \geq \frac{4(a - 1)}{a^2}\frac{\pi^2}{D^2} + \frac{2(a - 1)}{a}\Big{(}\kappa_2(m-1 + 2\kappa_1\Big{)}.
\end{align*}
Lastly, let $s = 1 - \frac{1}{a}$.  Then, substituting $s$ into the above inequality, we get
\begin{equation}\label{mainineq}
\mu \geq 4s(s-1)\frac{\pi^2}{D^2} + 2s(\kappa_2(m-1) + 2\kappa_1).
\end{equation}
Note that $s$ is an arbitrary number in the interval $(0,1)$ by our definition of $s$ and by the fact that $a>1$ was chosen arbitrarily.  Then \eqref{mainineq} holds for all such $s$.  Therefore
\begin{align*}
    \mu \geq \sup_{s\in (0,1)}\left\{4s(s-1)\frac{\pi^2}{D^2} + 2s(\kappa_2(m-1) + 2\kappa_1)\right\}.
\end{align*}

\end{document}